\documentclass[12pt]{article}
\usepackage{amsmath}
\usepackage{amssymb}
\usepackage{amsthm}
\begin{document}
\newtheorem{theorem}{Theorem}
\newtheorem{lemma}[theorem]{Lemma}
\newtheorem{corollary}[theorem]{Corollary}
\newtheorem{definition}[theorem]{Definition}
\newtheorem{example}[theorem]{Example}
\pagenumbering{roman}
\renewcommand{\thelemma}{\thesection.\arabic{lemma}}
\renewcommand{\thedefinition}{\thesection.\arabic{definition}}
\renewcommand{\theexample}{\thesection.\arabic{example}}
\renewcommand{\theequation}{\arabic{equation}}
\newcommand{\mysection}[1]{\section{#1}\setcounter{equation}{0}
\setcounter{theorem}{0} \setcounter{lemma}{0}
\setcounter{definition}{0}}
\newcommand{\mrm}{\mathrm}
\newcommand{\beq}{\begin{equation}}
\newcommand{\eeq}{\end{equation}}

\newcommand{\ben}{\begin{enumerate}}
\newcommand{\een}{\end{enumerate}}

\newcommand{\beqa}{\begin{eqnarray}}
\newcommand{\eeqa}{\end{eqnarray}}

\newcommand {\non}{\nonumber}
\newcommand{\C}{\mbox{$\mathbb{C}$}}
\newcommand{\rank}{\text{rank}}
\title
{\bf Analogue of the Cauchy--Schwarz Inequality for Determinants:\\ A Simple Proof }
\author
{Avram Sidi\\
Computer Science Department\\
Technion - Israel Institute of Technology\\ Haifa 32000, Israel\\
e-mail:  asidi@cs.technion.ac.il\\
https://asidi.cswp.cs.technion.ac.il/}
\date{\today}
\maketitle \thispagestyle{empty}
\newpage
\begin{abstract}\sloppypar
In this note, we present a simple proof of an analogue of the Cauchy--Schwarz inequality relevant to
products of determinants. Specifically, we show that
$$ |\det(A^*MB)|^2\leq \det(A^*MA)\cdot \det(B^*MB),\quad A,B\in \C^{m\times n},$$
where $M\in\C^{m\times m}$ is hermitian positive definite. Here $m$ and $n$ are arbitrary.
In case $m\leq n$, equality holds trivially.
Equality holds when $m>n$ and $\text{rank}(A)=\text{rank}(B)=n$ if and only if the columns   of $A$ and the columns of $B$ span the same subspace of $\C^m$.

\end{abstract}
\vspace{5cm} \noindent {\bf Mathematics Subject Classification
2020:} 15A15

\vspace{1cm} \noindent {\bf Keywords and expressions:} Cauchy--Schwarz inequality, determinants.

 \thispagestyle{empty}
\newpage
\pagenumbering{arabic}
\setcounter{equation}{0}
\setcounter{theorem}{0}

The Cauchy--Schwarz inequality for vectors in $\C^n$ states that, if $x,y\in \C^n$,
and $(a,b)$ is the inner product in $\C^n$, then
$$ |(x,y)|^2\leq (x,x)\cdot (y,y),$$ with equality if and only if $x$ and $y$ are linearly dependent.
 There are many analogues of this theorem in different settings, and these can be found in
many books and papers on linear algebra and related subjects. In this work, we prove an analogue of this inequality that concerns  determinants. Our result here also follows from
Marcus and Moore \cite{Marcus:1981:DFC}. (See the remarks at the end of this note.)

We begin with the following lemma:

\begin{lemma}\label{le} Let $U$ and $V$ be two rectangular unitary matrices in $\C^{m\times n}$ with $m>n$, in the sense
$$U^*U=I_n,\quad V^*V=I_n.$$
Then
\beq \label{eqer} |\det(U^*V)|\leq 1.\eeq
Equality holds if and only if the columns of $U$ and the columns of $V$ span the same subspace of $\C^m$.
\end{lemma}
\begin{proof}
The matrices $U$ and $V$ have the following columnwise partitionings:
$$ U=[u_1|u_2|\cdots|u_n],\quad u_i^*u_j=\delta_{i,j};\quad
V=[v_1|v_2|\cdots|v_n],\quad v_i^*v_j=\delta_{i,j}.$$
Then $W=U^*V\in \C^{n\times n}$ and the $(i,j)$ element of  $W$ is $u_i^*v_j$. Therefore, $U^*V$ has the columnwise partitioning
\beq \label{eqzx} U^*V=W=[w_1|w_2|\cdots|w_n];\quad w_j=[u_1^*v_j,u_2^*v_j,\ldots,u_n^*v_j]^T,\quad j=1,2,\dots,n, \eeq
and
\beq \label{eqbn} \|w_j \|^2=\sum_{i=1}^n|u_i^*v_j|^2,\quad j=1,2,\ldots,n.\footnote{We denote by $\|z\|$, $z\in \C^q$, the standard $l_2$ norm of $z$, in every dimension $q$. Thus $\|z\|=\sqrt{z^*z}$ for all $q$.}\eeq
Thus, by Hadamard's inequality,\footnote{Hadamard's inequality: If $H$ is an $s\times s$ matrix with columnwise partitioning $H=[h_1|h_2|\cdots|h_s]$, then $$|\det H|\leq \prod^s_{j=1}\|h_j\|,\quad
     \|h_j\|=\sqrt{h_j^*h_j},\quad j=1,\ldots,s.$$
      A proof of Hadamard's inequality,  can be found in Horn and Johnson \cite[p. 477]{Horn:1985:MA}.} there holds
   \beq \label{eqcv} |\det (U^*V)|=|\det W|\leq \prod^n_{j=1}\|w_j\|.\eeq

Let us now add to the set of the $n$  (orthonormal) vectors  $u_1,u_2,\ldots,u_n$ the $m-n$  (orthonormal) vectors $u_{n+1},u_{n+2},\ldots, u_m$, such that
$u_1,\ldots,u_m$ is an orthonormal basis for $\C^m$, hence for every vector $x\in \C^m$,  we have
$$x=\sum^m_{i=1} (u_i^*x)u_i\quad \text{and}\quad \|x\|^2=\sum^m_{i=1} |u_i^*x|^2.$$
Next, let us  define
 $$ S_U=\text{span}\{u_1,u_2,\ldots,u_n\} \quad  \text{and}\quad
   S_U^{\perp}=\text{span}\{u_{n+1},u_{n+2},\ldots,u_m\}.$$
 Here $S_U$ is the column space of  $U$ and $S^{\perp}_U$ is  the orthogonal complement of $S_U$, and every vector $x\in \C^m$ is of the form
  $$x=\hat{x}+\tilde{x};\quad \hat{x}=\sum^n_{i=1} (u_i^*x)u_i\in S_U,\quad \tilde{x}=\sum^m_{i=n+1} (u_i^*x)u_i\in S^{\perp}_U;\quad
  \hat{x}^*\tilde{x}=0,$$
  therefore,$$ \|x\|^2=\|\hat{x}\|^2+\|\tilde{x}\|^2;\quad
  \|\hat{x}\|^2=\sum^n_{i=1} |u_i^*x|^2,\quad \|\tilde{x}\|^2=\sum^m_{i=n+1} |u_i^*x|^2.$$
  Then, for
 $v_j$, the $j^{\text{th}}$ column of the matrix $V$,  we have
 $ v_j=\hat{v}_j+\tilde{v}_j$, where
$$\|\hat{v}_j\|^2=\sum^n_{i=1} |u_i^*v_j|^2=\|w_j\|^2\quad \text{by \eqref{eqbn}},\quad \|\tilde{v}_j\|^2=\sum^m_{i=n+1}|u_i^*v_j|^2.$$
Therefore,
 \beq \label{eqkl} 1=\|v_j\|^2=\|w_j\|^2+\|\tilde{v}_j\|^2\quad \Rightarrow\quad \|w_j\|\leq \|v_j\|=1;\quad j=1,2,\ldots,n.\eeq
This forces $\prod^n_{j=1}\|w_j\|\leq 1$ in  \eqref{eqcv}, thus proving \eqref{eqer}.

 Let us now denote the column space of $V$ by $S_V$. That is,
$$ S_V=\text{span}\{v_1,v_2,\ldots,v_n\}.$$
and let us assume that $S_V\neq S_U$. Then,  at least one of the vectors $v_1,v_2,\ldots, v_n$, say $v_p$,  does not belong to $S_U$, and this implies that
$$v_p=\hat{v}_p+\tilde{v}_p,\quad \tilde{v}_p\neq0.$$
As a result,
\beq\label{eqmn} 1=\|v_p\|^2=\|w_p\|^2+\|\tilde{v}_p\|^2> \|w_p\|^2\quad \Rightarrow\quad
\|w_p\|<\|v_p\|=1 .\eeq
As a result of \eqref{eqkl} and \eqref{eqmn}, we have
$\prod^n_{j=1}\|w_j\|<1$ in  \eqref{eqcv},
which forces  strict inequality in  \eqref{eqer}.

We now show that equality holds in \eqref{eqer} when  $S_U=S_V$.
In this case,  each $v_i$ is a linear combination of the $u_j$. That is,
$$V= U\Sigma,\quad \text{for some $\Sigma\in \C^{n\times n}$}.$$
Consequently,
$$ I_n=V^*V=(U\Sigma)^*(U\Sigma)=\Sigma^*(U^*U)\Sigma=\Sigma^*\Sigma\quad \Rightarrow\quad
\Sigma^*\Sigma=I_n,$$
that is, $\Sigma$ is a unitary matrix in $\C^{n\times n}$. Consequently,
$$ U^*V=U^*(U\Sigma)=(U^*U)\Sigma=\Sigma\quad \Rightarrow\quad \det(U^*V)=\det \Sigma,$$
from which,
$$|\det(U^*V)|=|\det \Sigma|=1. $$

This completes the proof.
\end{proof}

\begin{theorem}\label{th1} Let $A,B\in \C^{m\times n}$, with $m,n$ arbitrary. Then
\beq \label{eq1} |\det(A^*B)|^2\leq \det(A^*A)\cdot \det(B^*B).\eeq
\begin{enumerate}
\item
 If  $m<n$,  equality holds in \eqref{eq1}, both sides vanishing there.
\item
If $m=n$, equality holds in \eqref{eq1}.
\item
(a)\,If $m>n$ and $\text{\em rank}(A)<n$ or
 $\text{\em rank}(B)<n$, equality holds in  \eqref{eq1}, both sides vanishing there.

 (b)\,If $m>n$ and $\text{\em rank}(A)=\text{\em rank}(B)=n$, equality holds in  \eqref{eq1} if and only if
 the columns of $A$ and the columns of $B$ span the same subspace of $\C^m$.
 \end{enumerate}
\end{theorem}

\begin{proof} We start by noting that all three matrices $A^*A$, $B^*B$, and $A^*B$ are $n\times n$.

\begin{enumerate}
\item
If $m<n$, each one of the   matrices $A^*A$, $B^*B$, and $A^*B$ is of rank at most $m$, hence is  singular.
Therefore,  \eqref{eq1} holds, both sides vanishing there.
\item If $m=n$, then $A$ and $B$ are square. Therefore,
    $$\det(A^*A)=(\det A^*)(\det A),\quad \det(B^*B)=(\det B^*)(\det B) $$ and
    $$ \det(A^*B)=(\det A^*)(\det B).$$
    The result in \eqref{eq1} now follows with equality there by invoking $\det C^*=\overline{\det C}$ for every square matrix $C$.
\item (a)\,If $m>n$,  and either $\text{rank}(A)<n$ (hence $A^*A$ is singular) or $\text{rank}(B)<n$ (hence $B^*B$ is singular), we
    have that $A^*B$ is singular as well.
    Therefore,  \eqref{eq1} holds with equality, both sides vanishing there.\\ \\
    (b)\,If $m>n$ and $\text{rank}(A)=\text{rank}(B)=n$, we proceed as follows:\\
   Consider the QR factorizations of $A$ and $B$, namely,
      $$ A=Q_AR_A,\quad B=Q_BR_B;\quad Q_A,Q_B\in \C^{m\times n},\quad R_A,R_B\in\C^{n\times n},$$
      where $Q_A$ and $Q_B$ are unitary in the sense that
      $$Q_A^*Q_A=I_n,\quad Q_B^*Q_B=I_n,$$ and $R_A$ and $R_B$ are upper triangular square matrices with nonzero diagonal elements. Now, it is easy to verify that
      $$ A^*A=R_A^*R_A,\quad B^*B=R_B^* R_B,\quad A^*B=R_A^*(Q_A^*Q_B)R_B.$$
      Note that $Q_A^*Q_B\in \C^{n\times n}$, just as $R_A,R_B$.
      Therefore, since $R_A$ and $R_B$ are upper triangular matrices with nonzero diagonal elements,
      $$  \det(A^*A)=(\det R_A^*)( \det R_A)=| \det R_A|^2,$$ and
      $$\det(B^*B)=(\det R_B^*) (\det R_B)=| \det R_B|^2.$$
      Next,
      $$ \det(A^*B)=(\det(Q_A^*Q_B))(\det R_A^*) (\det R_B).$$
      Therefore,
      $$
    |\det(A^*B)|=|\det(Q_A^*Q_B)| \,|\det  R_A|\, |\det R_B|,$$ 
     which implies that

     $$|\det(A^*B)|^2=|\det(Q_A^*Q_B)|^2 \det(A^*A) \det(B^*B).$$
    The rest of the proof can now be achieved   (i)\,by realizing that the subspaces spanned by $A$ and $Q_A$  are the same,
    and so are the subspaces spanned by  $B$ and $Q_B$, and (ii)\,by  applying Lemma~\ref{le} to   $|\det(Q_A^*Q_B)|$ since $m>n$. We leave the details to the reader.
 \end{enumerate}
 This completes the proof.
\end{proof}

By applying Theorem \ref{th1} to the matrices $\tilde{A}=M^{1/2}A$ and $\tilde{B}=M^{1/2}B$, where $M\in\C^{m\times m}$ is hermitian positive definite, we obtain the following general form of Theorem \ref{th1}.

\begin{theorem}
Let $A,B\in \C^{m\times n}$, with $m,n$ arbitrary, and let  $M\in\C^{m\times m}$ be hermitian positive definite. Then
\beq \label{eq2} |\det(A^*MB)|^2\leq \det(A^*MA)\cdot \det(B^*MB).\eeq
\begin{enumerate}
\item
 If  $m<n$,  equality holds in \eqref{eq2}, both sides vanishing there.
\item
If $m=n$, equality holds in \eqref{eq2}.
\item
(a)\,If $m>n$ and $\text{\em rank}(A)<n$ or
 $\text{\em rank}(B)<n$, equality holds in  \eqref{eq2}, both sides vanishing there.

 (b)\,If $m>n$ and $\text{\em rank}(A)=\text{\em rank}(B)=n$, equality holds in  \eqref{eq2} if and only if
 the columns of $A$ and the columns of $B$ span the same subspace of $\C^m$.
 \end{enumerate}
\end{theorem}

\noindent{\bf Remarks.}
\begin{enumerate}
\item
 The proof of the Cauchy--Schwarz inequality in $\C^n$ makes use of the fact that the inner product $(x,y)$ in $\C^n$  is bilinear  in $x$ and $y$. Because  $X^*MY$   is bilinear in $X$ and $Y$, one might think that the proof of the
Cauchy--Schwarz inequality in $\C^n$ can be applied to $\det(A^*MB)$, $\det(A^*MA)$, and
$\det(B^*MB)$ to obtain \eqref{eq2}. This is not possible, however, since $\det(A^*MB)$ does not have the bilinearity property; for example, $\det((A_1+A_2)^*MB)$ is not necessarily equal to $\det(A_1^*MB)+\det(A_2^*MB).$
\item
The problem treated here is a special case of a general problem  treated by Marcus and Moore in \cite[Eq. (1) on p. 111 and Theorem on p. 115] {Marcus:1981:DFC}. These authors  use  advanced  techniques to obtain the relations which must exist
between the $m\times m$ matrices $M_1,M_2,M_3,M_4$  so that the relation
$$ \det( A^*M_1B) \det( B^*M_2A) \leq \det( A^*M_3A) \det( B^*M_4B) $$
holds for all $m\times n$ matrices $A$ and $B$ when $m>n$.

Because we are restricting our problem to the special case in  which $M_1=M_2=M_3=M_4=M$,
$M$ being an $m\times m$ positive definite hermitian matrix, we are able to carry out our analysis by employing rather elementary techniques of linear algebra that are easily accessible.
\end{enumerate}


\begin{thebibliography}{1}

\bibitem{Horn:1985:MA}
R.A. Horn and C.R. Johnson.
\newblock {\em {Matrix Analysis}}.
\newblock Cambridge University Press, Cambridge, 1985.

\bibitem{Marcus:1981:DFC}
M.~Marcus and K.~Moore.
\newblock A determinant formulation of the {Cauchy-Schwarz} inequality.
\newblock {\em Linear Algebra Appl.}, 36:111--127, 1981.

\end{thebibliography}

\end{document}